\documentclass[final,leqno]{siamltex}

\usepackage{amsmath,amssymb}
\usepackage{graphics,graphicx}

\newcommand{\bq}{\begin{equation}}
\newcommand{\eq}{\end{equation}}
\renewcommand{\ldots}{\dotsc}

\newcommand{\bv}{{\bf v}}

\newcommand{\bn}{{\bf n}}
\newcommand{\bt}{{\bf t}}
\newcommand{\vn}{{\bf n}}

\def\bc{{\bf c}}

\def\bn{{\bf n}}
\def\bq{{\bf q}}

\def\3bar{{|\hspace{-.02in}|\hspace{-.02in}|}}

\newcommand{\bbr}{\mathbb{R}}
\newcommand{\vbeta}{\boldsymbol{\beta}}
\newcommand{\vchi}{\boldsymbol{\chi}}
\renewcommand{\div}{\nabla\cdot}

\newcommand{\calF}{{\cal F}}
\newcommand{\calT}{{\cal T}}
\newcommand{\ve}{{\bf e}}
\newcommand{\vp}{{\bf p}}
\newcommand{\vq}{{\bf q}}
\newcommand{\calA}{{\cal A}}

\begin{document}

\title{A Computational Study of the Weak Galerkin Method for Second-Order Elliptic Equations}

\author{
Lin Mu
\thanks{Department of Applied Science,
University of Arkansas at Little Rock,
Little Rock, AR 72204, USA,
\texttt{lxmu@ualr.edu}}
\and
Junping Wang
\thanks{Division of Mathematical Sciences,
National Science Foundation, Arlington, VA 22230, USA,
\texttt{jwang@nsf.gov}. The research of Wang was supported by the
NSF IR/D program, while working at the Foundation. However, any
opinion, finding, and conclusions or recommendations expressed in
this material are those of the author and do not necessarily reflect
the views of the National Science Foundation.} \and Yanqiu Wang
\thanks{Department of Mathematics,
Oklahoma State University,
Stillwater, OK 74078, USA,
\texttt{yqwang@math.okstate.edu}}
\and
Xiu Ye
\thanks{Department of Mathematics,
University of Arkansas at Little Rock, Little Rock, AR 72204, USA,
\texttt{xxye@ualr.edu}. The research of Ye was supported in part by
the National Science Foundation under Grant No. DMS-1115097.}}

\maketitle

\begin{abstract}
The weak Galerkin finite element method is a novel numerical method
that was first proposed and analyzed by Wang and Ye in
\cite{WangYe_PrepSINUM_2011} for general second order elliptic
problems on triangular meshes. The goal of this paper is to conduct
a computational investigation for the weak Galerkin method for
various model problems with more general finite element partitions.
The numerical results confirm the theory established in
\cite{WangYe_PrepSINUM_2011}. The results also indicate that the
weak Galerkin method is efficient, robust, and reliable in
scientific computing.
\end{abstract}



%%%%%%%%%%%%%%%%%%%%%%%%%%%%%%%%%%%%%%%%%%%%%%%%%%%%%%%%%%%
\section{Introduction}
In this paper, we are concerned with computation and numerical
accuracy issues for the {\em weak Galerkin} method that was recently
introduced in \cite{WangYe_PrepSINUM_2011} for second order elliptic
equations. The weak Galerkin method is an extension of the standard
Galerkin finite element method where classical derivatives were
substituted by weakly defined derivatives on functions with
discontinuity. The weak Galerkin method is also related to the
standard mixed finite element method in that the two methods are
identical for simple model problems (such as the Poisson problem).
But they have fundamental differences for general second order
elliptic equations. The goal of this paper is to numerically
demonstrate the efficiency and accuracy of the weak Galerkin method
in scientific computing. In addition, we shall extend the weak
Galerkin method of \cite{WangYe_PrepSINUM_2011} from triangular and
tetrahedral elements to rectangular and cubic elements.

For simplicity, we take the linear second order elliptic equation as
our model problem. More precisely, let $\Omega$ be an open bounded
domain in $\bbr^d$, $d = 2,3$ with Lipschitz continuous boundary
$\partial\Omega$. The model problem seeks an unknown function
$u=u(x)$ satisfying
\begin{equation} \label{eq:ellipticproblem}
  \begin{array}{rll}
    -\nabla\cdot(\calA \nabla u) + \vbeta \cdot\nabla u + \gamma u &=& f\qquad \textrm{in }
    \Omega, \\
    u &=& g \qquad \textrm{on } \partial\Omega,
  \end{array}
\end{equation}
where $\calA\in [L^{\infty}(\Omega)]^{d\times d}$, $\vbeta\in
[L^{\infty}(\Omega)]^d$, and $\gamma\in L^{\infty}(\Omega)$ are
vector- and scalar-valued functions, as appropriate. Furthermore,
assume that $\calA$ is a symmetric and uniformly positive definite
matrix and the problem (\ref{eq:ellipticproblem}) has one and only
one weak solution in the usual Sobolev space $H^1(\Omega)$
consisting of square integrable derivatives up to order one. $f$ and
$g$ are given functions that ensure the desired solvability of
(\ref{eq:ellipticproblem}).

Throughout the paper, we use $\|\cdot\|$ to denote the standard
$L^2$ norm over the domain $\Omega$, and use bold face Latin
characters to denote vectors or vector-valued functions. The paper
is organized as follows. In Section 2, the weak Galerkin method is
introduced and an abstract theory is given. In particular, we prove
that certain rectangular elements satisfy the assumptions in the
abstract theory, and thus establish a well-posedness and error
estimate for the corresponding weak Galerkin method with rectangular
meshes. In Section 3, we present some implementation details for the
weak Galerkin elements. Finally in Section 4, we report some
numerical results for various test problems. The numerical
experiments not only confirm the theoretical predictions as given in
the original paper \cite{WangYe_PrepSINUM_2011}, but also reveal new
results that have not yet been theoretically proved.

%%%%%%%%%%%%%%%%%%%%%%%%%%%%%%%%%%%%%%%%%%%%%%%%%%%%%%%%%%%%%%%%%%%%%%%%%%%%%%%%%%%%
%%%%%%%%%%%%%%%%%%%%%%%%%%%%%%%%%%%%%%%%%%%%%%%%%%%%%%%%%%%%%%%%%%%%%%%%%%%%%%%%%%%%
%%%%%%%%%%%%%%%%%%%%%%%%%%%%%%%%%%%%%%%%%%%%%%%%%%%%%%%%%%%%%%%%%%%%%%%%%%%%%%%%%%%%
\section{The Weak Galerkin Method}

Let ${\cal T}_h$ be a shape-regular, quasi-uniform mesh of the
domain $\Omega$, with characteristic mesh size $h$. In
two-dimension, we consider triangular and rectangular meshes, and in
three-dimension, we mainly consider tetrahedral and hexahedral
meshes. For each element $K\in {\cal T}_h$, denote by $K_0$ and
$\partial K$ the interior and the boundary of $K$, respectively.
Here, $K$ can be a triangle, a rectangle, a tetrahedron or a
hexahedron. The boundary $\partial K$ consists of several ``sides'',
which are edges in two-dimension or faces(polygons) in
three-dimension. Denote by $\calF_h$ the collection of all
edges/faces in ${\cal T}_h$.

On each $K\in {\cal T}_h$, let $P_j(K_0)$ be the set of polynomials
on $K_0$ with degree less than or equal to $j$, and $Q_j(K)$ be the
set of polynomials on $K_0$ with degree of each variable less than
or equal to $j$. Likewise, on each $F\in \calF_h$, $P_l(F)$ and
$Q_l(F)$ are defined analogously. Now, define a weak discrete space
on mesh $\calT_h$ by
$$
\begin{aligned}
S_{h} = & \{v:\:  v|_{K_0}\in P_j(K_0) \textrm{ or }Q_j(K_0) \textrm{ for all } K\in \calT_h,\\
 &\qquad  v|_F\in P_l(F) \textrm{ or }Q_l(F)  \textrm{ for all } F\in \calF_h \}.
\end{aligned}
$$
Observe that the definition of $S_h$ does not require any form of
continuity across element or edge/face interfaces. A function in
$S_h$ is characterized by its value on the interior of each element
plus its value on the edges/faces. Therefore, it is convenient to
represent functions in $S_h$ with two components, $v=\{v_0, v_b\}$,
where $v_0$ denotes the value of $v$ on all $K_0$s and $v_b$ denotes
the value of $v$ on $\calF_h$.

We further define an $L^2$ projection from $H^1(\Omega)$ onto $S_h$
by setting $Q_h v \equiv \{Q_0 v,\, Q_b v\}$, where $Q_0 v|_K$ is
the local $L^2$ projection of $v$ in $P_j(K_0)$, for $K\in\calT_h$,
and $Q_b v|_F$ is the local $L^2$ projection in $P_l(F)$, for $F\in
\calF_h$.

The idea of the weak Galerkin method is to seek an approximate
solution to Equation (\ref{eq:ellipticproblem}) in the weak discrete
space $S_h$. To this end, we need to introduce a discrete gradient
operator on $S_h$. Indeed, this will be done locally on each element
$K$. Let $V_r(K)$ be a space of polynomials on $K$ such that
$[P_{r}(K)]^d\subset V_r(K)$; details of $V_r(K)$ will be given
later. Let
$$
\Sigma_h = \{\vq\in [L^2(\Omega)]^d:\: \vq|_K \in V_r(K)\textrm{ for
all }K\in \calT_h\}.
$$
A discrete gradient of $v_h=\{v_0,v_b\}\in S_h$ is defined to be a
function $\nabla_d v_h \in \Sigma_h$ such that on each
$K\in\calT_h$,
\begin{equation}\label{discrete-weak-gradient-new}
\int_K \nabla_{d} v_h\cdot \vq\, dx = -\int_K v_0 \div \vq\, dx + \int_{\partial K} v_b \vq\cdot\bn\, ds,
\quad \textrm{for all } \vq\in V_r(K),
\end{equation}
where $\bn$ is the unit outward normal on $\partial K$.
Clearly, such a discrete gradient is always well-defined.

Denote by $(\cdot,\cdot)$ the standard $L^2$-inner product on
$\Omega$. Let $S_h^0$ be a subset of $S_h$ consisting of functions
with vanishing boundary values. Now we can write the weak Galerkin
formulation for Equation (\ref{eq:ellipticproblem}) as follows: find
$u_h = \{u_0, u_b\} \in S_h$ such that $u_b= Q_b g$ on each
edge/face $F\subset \partial\Omega$ and
\begin{equation} \label{eq:weakformulation}
(\calA\nabla_d u_h, \nabla_d v_h)  + (\vbeta \cdot \nabla_d u_h,
v_{0}) + (\gamma u_{0},\,v_{0}) = (f, v_{0})
\end{equation}
for all $v_h = \{v_0, v_b\} \in S_h^0$. For simplicity of notation,
we introduce the following bilinear form
\begin{equation}\label{bilinearform}
a(u_h, v_h)\triangleq (\calA\nabla_d u_h, \nabla_d v_h)  + (\vbeta
\cdot \nabla_d u_h, v_{0}) + (\gamma u_{0},\,v_{0}).
\end{equation}

The spaces $S_h$ and $\Sigma_h$ can not be chosen arbitrarily. There
are certain criteria they need to follow, in order to guarantee that
Equation (\ref{eq:weakformulation}) provides a good approximation to
the solution of Equation (\ref{eq:ellipticproblem}). For example,
$\Sigma_h$ has to be rich enough to prevent from the loss of
information in the process of taking discrete gradients, while it
should remain to be sufficiently small for its computational cost.
Hence, we would like to impose the following conditions upon $S_h$
and $\Sigma_h$: \vskip 0.2cm

\begin{description}
  \item{{\bf (P1) }} For any $v_h\in S_h$ and $K\in\calT_h$, $\nabla_d v_h|_K = 0$ if and only if $v_0=v_b = constant$ on $K$.
  \item{{\bf (P2) }} For any $w\in H^{m+1}(\Omega)$, where $0\le m\le j+1$, we have
    $$\|\nabla_d (Q_h w)-\nabla w\|\le Ch^{m}\|w\|_{m+1},$$
   where and in what follows of this paper, $C$ denotes a generic constant independent of
   the mesh size $h$.
\end{description}
\medskip

Under the above two assumptions, it has been proved in
\cite{WangYe_PrepSINUM_2011} that Equation
(\ref{eq:weakformulation}) has a unique solution as long as the mesh
size $h$ is moderately small and the dual of
(\ref{eq:ellipticproblem}) has an $H^{1+s}$-regularity with some
$s>0$. Furthermore, one has the following error estimate:
\begin{equation} \label{eq:errorestimation}
\begin{aligned}
\|\nabla_d (u_h-Q_h u)\| &\le C\left(h^{1+s} \|f-Q_0 f\| + h^{m}\|u\|_{m+1}\right), \\
\|u_0-Q_0 u\| &\le C\left(h^{1+s} \|f-Q_0 f\| + h^{m+s}\|u\|_{m+1}\right),
\end{aligned}
\end{equation}
for any $0\le m\le j+1$, and $s>0$ is the largest number such that
the dual of Equation (\ref{eq:ellipticproblem}) has an
$H^{1+s}$-regularity.

\medskip

There are several possible combinations of $S_h$ and $\Sigma_h$ that
satisfy Assumptions {\bf (P1)} and {\bf (P2)}. Two examples of
triangular elements have been given in \cite{WangYe_PrepSINUM_2011},
which are
\begin{enumerate}
\item Triangular element $(P_j(K_0),\,P_j(F),\, RT_j(K))$ for $j\ge 0$.
That is, in the definition of $S_h$, we set $l=j$. And in the definition
of $\Sigma_h$, we set $r=j$ and $V_r(K)$ to be the $j$th order Raviart-Thomas element $RT_j(K)$ \cite{rt}.
\item Triangular element $(P_j(K_0),\,P_{j+1}(F),\, (P_{j+1}(K))^d)$ for $j\ge 0$.
That is, in the definition of $S_h$, we set $l=j+1$.
And in the definition of  $\Sigma_h$, we set $r=j+1$ and $V_r(K)=(P_{j+1}(K))^d$, or in other words, the $(j+1)$st order Brezzi-Douglas-Marini element \cite{bdm}.
\end{enumerate}

\medskip
The rest of this section shall extend this result to rectangular
elements. An extension to three-dimensional tetrahedral and
hexahedral elements is straightforward.

%%%%%%%%%%%%%%%%%%%%%%%%%%%%%%%%%%%%%%%%%%%%%%%%%%%%%%%%%%%%%%%%%%%%%%%
\subsection{Weak Galerkin on Rectangular Meshes}
Consider the following two type of rectangular elements:
\begin{enumerate}
\item Rectangular element $(Q_j(K_0),\,Q_j(F),\, RT_j(K))$ for $j\ge 0$.
That is, in the definition of $S_h$, we set $l=j$. And in the definition
of $\Sigma_h$, we set $r=j$ and $V_r(K)$ to be the $j$th order Raviart-Thomas element $RT_j(K)$ on rectangle $K$.
\item Rectangular element $(P_j(K_0),\,P_{j+1}(F),\, BDM_{j+1}(K))$ for $j\ge 0$.
That is, in the definition of $S_h$, we set $l=j+1$.
And in the definition of  $\Sigma_h$, we set $r=j+1$ and $V_r(K)$ to be the $(j+1)$st order Brezzi-Douglas-Marini element $BDM_{j+1}(K)$
on rectangle $K$.
\end{enumerate}\medskip
Denote by $Q_{i,j}(K)$ the space of polynomials with degree in $x$
and $y$ less than or equal to $i$ and $j$, respectively, and
${\bf curl} = \begin{bmatrix}-\partial/\partial y\\
\partial/\partial x\end{bmatrix}$. It is known that
$$
\begin{aligned}
  RT_j(K) &= \begin{bmatrix}Q_{j+1,j}(K) \\ Q_{j,j+1}(K) \end{bmatrix}, \\[2mm]
  BDM_{j+1}(K) &= \begin{bmatrix}P_{j+1}(K)\\ P_{j+1}(K)  \end{bmatrix}
           + span\left\{{\bf curl}\, (x^{j+2}y), \; {\bf curl}\, (xy^{j+2}) \right\},
\end{aligned}
$$
and $\dim(RT_j(K)) = 2(j+1)(j+2)$, $\dim(BDM_{j+1}(K)) = (j+2)(j+3)+2$.
The degrees of freedom for $RT_j(K)$ are:
$$
\begin{aligned}
&\int_F (\vq\cdot\vn) w\, ds ,\qquad && \textrm{for all }w\in Q_j(F),\;F\in\calF\cap\partial K,\\[2mm]
&\int_K \vq\cdot\vp\, dx,\qquad && \textrm{for all }\vp\in Q_{j-1,j}(K)\times Q_{j,j-1}(K).
\end{aligned}
$$
The degrees of freedom for $BDM_{j+1}(K)$ are
$$
\begin{aligned}
&\int_F (\vq\cdot\vn) w\, ds ,\qquad && \textrm{for all }w\in P_{j+1}(F),\;F\in\calF\cap\partial K,\\[2mm]
&\int_K \vq\cdot\vp\, dx,\qquad && \textrm{for all }\vp\in
[P_{j-1}(K)]^2.
\end{aligned}
$$
It is also well-known that on each rectangle $K\in \calT_h$ and each edge $F\in\calF_h\cap\partial K$,
\begin{equation} \label{eq:rec2}
\begin{aligned}
\nabla\cdot RT_j(K) &= Q_j(K_0),\qquad & RT_j(K)\cdot\vn|_F &= Q_j(F),\\
\nabla\cdot BDM_{j+1}(K) &= P_j(K_0),\qquad & BDM_{j+1}(K)\cdot\vn|_F &= P_{j+1}(F).
\end{aligned}
\end{equation}

Next, we show that the two set of elements defined as above satisfy
Assumptions {\bf (P1)} and {\bf (P2)}.

\medskip
\begin{lemma}\label{lemma2.1}
  For the two type of rectangular elements given in this subsection, the Assumption
  {\bf P1} holds true.
\end{lemma}
\begin{proof}
If $v_0=v_b = constant$ on $K$, then clearly $\nabla_d v_h|_K$
vanishes since the right-hand side of
(\ref{discrete-weak-gradient-new}) is zero from the divergence
theorem. Now let us assume that $\nabla_d v_h|_K = 0$. By
(\ref{discrete-weak-gradient-new}) and using integration by parts,
we have for all $\vq\in  RT_j(K)$ or $BDM_{j+1}(K)$,
\begin{equation} \label{eq:rec1}
\begin{aligned}
0 &= -\int_K v_0 \div \vq\, dx + \int_{\partial K} v_b \vq\cdot\bn\, ds\\
&= \int_{\partial K} (v_b-v_0)\vq\cdot\bn\, ds + \int_K (\nabla v_0)\cdot \vq\, dx.
\end{aligned}
\end{equation}

We first consider the element $(Q_j(K_0),\,Q_j(F),\, RT_j(K))$. If
$j=0$, then $v_0$ is a constant on $K_0$ and clearly $\nabla
v_0={\bf 0}$. If $j>0$, take $\vq$ such that $\int_F (\vq\cdot\vn)
w\, ds=0$ for all $w\in Q_j(F)$ and let it traverse through all
degrees of freedom defined by $\int_K \vq\cdot\vp\, dx$, for $\vp\in
Q_{j-1,j}(K)\times Q_{j,j-1}(K)$. Since $(v_b-v_0)|_F \in Q_j(F)$
and $\nabla v_0\in Q_{j-1,j}(K)\times Q_{j,j-1}(K)$, Equation
(\ref{eq:rec1}) gives $\nabla v_0 = {\bf 0}$, which implies that
$v_0$ is a constant on $K_0$. Now Equation (\ref{eq:rec1}) reduces
into
$$
\int_{\partial K} (v_b-v_0)\vq\cdot\bn\, ds = 0,\qquad \textrm{for all }\vq\in  RT_j(K).
$$
Next, since $(v_b-v_0)|_F \in Q_j(F) = RT_j(K)\cdot\vn|_F$ for all $F\in \calF_h\cap \partial K$,
by letting $\vq$ traverse through all degrees of freedom on $\partial K$,
we have $v_b-v_0=0$ on all $F$.
This implies $v_b=v_0=constant$ in $K$.

For the $(P_j(K_0),\,P_{j+1}(F),\, BDM_{j+1}(K))$ element, using the same argument as in the previous case,
and noticing that $\nabla v_0\in (P_{j-1}(K))^2$, $(v_b-v_0)|_F \in P_{j+1}(F) = BDM_{j+1}(K)\cdot\vn|_F$ for all $F\in \calF_h\cap \partial K$,
we can similarly prove that $v_b=v_0=constant$ in $K$.
\end{proof}

\medskip
\begin{lemma}\label{lemma2.2}
  For the two type of rectangular elements given in this subsection, the Assumption {\bf (P2)}
  holds true.
\end{lemma}
\begin{proof}
  Let $w\in H^m(\Omega)$, $0\le m\le j+1$. For any $\vq\in \Sigma_h$ and $K\in\calT_h$,
by (\ref{eq:rec2}) and the definition of $L^2$ projections, we have
$$
  \begin{aligned}
    \int_K (\nabla_d Q_h w)\cdot\vq\, dx &= -\int_K(Q_0 w)(\nabla\cdot \vq)\, dx +
    \int_{\partial K} Q_bw(\vq\cdot\vn)\, ds \\
    &= -\int_K  w (\nabla\cdot \vq)\, dx + \int_{\partial K} w(\vq\cdot\vn)\, ds \\
    &= \int_K (\nabla w)\cdot \vq\,dx.
  \end{aligned}
$$
In other words, on each $K\in\calT_h$, $\nabla_d Q_hw$ is the $L^2$
projection of $\nabla w$ onto $RT_j(K)$ or $BDM_{j+1}(K)$. Thus, the
Assumption {\bf (P2)} follows immediately from the approximation
properties of the $L^2$ projection, and the fact that both $RT_j(K)$
and $BDM_{j+1}(K)$ contains the entire polynomial space
$[P_j(K)]^2$.
\end{proof}

\medskip
Using Lemma \ref{lemma2.1} and Lemma \ref{lemma2.2}, one can derive
the error estimate (\ref{eq:errorestimation}) for the rectangular
elements by following the argument presented in
\cite{WangYe_PrepSINUM_2011}. Details are left to interested readers
as an exercise.

%%%%%%%%%%%%%%%%%%%%%%%%%%%%%%%%%%%%%%%%%%%%%%%%%%%%%%%%%%%%%%%%%%%%%%%
%%%%%%%%%%%%%%%%%%%%%%%%%%%%%%%%%%%%%%%%%%%%%%%%%%%%%%%%%%%%%%%%%%%%%%%
%%%%%%%%%%%%%%%%%%%%%%%%%%%%%%%%%%%%%%%%%%%%%%%%%%%%%%%%%%%%%%%%%%%%%%%
\section{Computation of Local Stiffness Matrices}
Similar to the standard Galerkin finite element method, the weak
Galerkin method (\ref{eq:weakformulation}) can be implemented as a
matrix problem where the matrix is given as the sum of local
stiffness matrices on each element $K\in\calT_h$. Thus, a key step
in the computer implementation of the weak Galerkin is to compute
element stiffness matrices. The goal of this section is to
demonstrate ways of computing element stiffness matrices for various
elements introduced in the previous sections.

For a given element $K\in\calT_h$, let $\phi_{0,i}$,
$i=1,\ldots,N_0$, be a set of basis functions for $P_j(K_0)$ or
$Q_j(K_0)$, and $\phi_{b,i}$, $i=1,\ldots,N_b$, be a set of basis
functions for $\sum_{F\in
\partial K\cap \calF_h} P_l(F)$ or $\sum_{F\in \partial K\cap
\calF_h} Q_l(F)$ . Note that $\{\phi_{b,i}\}$ is the union of basis
functions from all edges/faces of element $K$. Then every
$v_h=\{v_0,v_b\}\in S_h$ has the following representation in $K$:
$$
%v_h|_K = \sum_{i=1}^{N_0} v^K_{0,i} \phi^K_{0,i} + \sum_{i=1}^{N_b} v^K_{b,i} \phi^K_{b,i}.
v_h|_K = \left\{\sum_{i=1}^{N_0} v_{0,i} \phi_{0,i},
\sum_{i=1}^{N_b} v_{b,i} \phi_{b,i}\right\}.
$$
On each $K$, the local stiffness matrix $M_K$ for Equation (\ref{eq:weakformulation})
can thus be written as a block matrix
\begin{equation}\label{eq:localstiffness}
M_K = \begin{bmatrix} M_{0,0} & M_{0,b} \\ M_{b,0} & M_{b,b}\end{bmatrix}
\end{equation}
where $M_{0,0}$ is an $N_0\times N_0$ matrix, $M_{0,b}$ is an $N_0\times N_b$ matrix,
$M_{b,0}$ is an $N_b\times N_0$ matrix,
and $M_{b,b}$ is an $N_b\times N_b$ matrix. These matrices are defined, respectively, by
$$
\begin{aligned}
M_{0,0} &= \begin{bmatrix}a(\phi_{0,j}, \phi_{0,i})_K\end{bmatrix}_{i,j}, \qquad &
M_{0,b} &= \begin{bmatrix}a(\phi_{b,j}, \phi_{0,i})_K\end{bmatrix}_{i,j}, \\
M_{b,0} &= \begin{bmatrix}a(\phi_{0,j}, \phi_{b,i})_K\end{bmatrix}_{i,j}, \qquad &
M_{b,b} &= \begin{bmatrix}a(\phi_{b,j}, \phi_{b,i})_K\end{bmatrix}_{i,j},
\end{aligned}
$$
where the bilinear form $a(\cdot,\cdot)$ is defined as in
(\ref{bilinearform}), and $i$, $j$ are the row and column indices,
respectively.

To compute each block of $M_K$, we first need to calculate the
discrete gradient operator $\nabla_d$. For convenience, denote the
local vector representation of $v_h|_K$ by
$$
\underline{\bv}_0 = \begin{bmatrix}v_{0,1} \\[1mm] v_{0,2} \\ \vdots\\[1mm] v_{0,N_0}\end{bmatrix},\qquad
\underline{\bv}_b = \begin{bmatrix} v_{b,1} \\[1mm] v_{b,2} \\ \vdots \\[1mm] v_{b,N_b} \end{bmatrix}.
$$
Let $\vchi_i$, $i=1,\ldots,N_V$, be a set of basis functions for
$V_r(K)$. Then, for every $\vq_h\in \Sigma_h$, its value on $K$ can
be expressed as
$$
\vq_h|_K = \sum_{i=1}^{N_V} q_i \vchi_i.
$$
Similarly, we denote the local vector representation of $\vq_h|_K$
by
$$
\underline{\bq} = \begin{bmatrix}q_1\\[1mm] q_2\\ \vdots\\[1mm] q_{N_V} \end{bmatrix},\qquad
$$
Then, by the definition of the discrete gradient (\ref{discrete-weak-gradient-new}),
given $v_h|_K$, we can compute the vector form of $\nabla_d v_h$ on $K$ by
\begin{equation} \label{eq:localwg}
D_K\underline{(\nabla_d v_h)} =  -Z_K \underline{\bv}_0 + T_K \underline{\bv}_b,
\end{equation}
where the $N_V\times N_V$ matrix $D_K$, the $N_V\times N_0$ matrix $Z_K$, and the $N_V\times N_b$ matrix $T_K$
are defined, respectively, by
\begin{equation} \label{eq:DK}
D_K = \begin{bmatrix} \int_K\vchi_1\cdot\vchi_1\, dx & \cdots & \int_K\vchi_1\cdot\vchi_{N_V}\, dx \\
          \cdots & \cdots & \cdots \\
          \int_K\vchi_{N_V}\cdot\vchi_1\, dx & \cdots & \int_K\vchi_{N_V}\cdot\vchi_{N_V}\, dx\end{bmatrix},
\end{equation}
$$
Z_K = \begin{bmatrix}\int_K(\div\vchi_1)\phi_{0,1}\, dx & \cdots & \int_K(\div\vchi_1)\phi_{0,N_0}\, dx \\
          \cdots & \cdots & \cdots \\
          \int_K(\div\vchi_{N_V})\phi_{0,1}\, dx & \cdots & \int_K(\div\vchi_{N_V})\phi_{0,N_0}\, dx\end{bmatrix},
$$
and
$$
T_K = \begin{bmatrix}\int_{\partial K}(\vchi_1\cdot\bn)\phi_{b,1}\, ds & \cdots
                     & \int_{\partial K}(\vchi_1\cdot\bn)\phi_{b,N_b}\, ds \\
          \cdots & \cdots & \cdots \\
          \int_{\partial K}(\vchi_{N_V}\cdot\bn)\phi_{b,1}\, ds & \cdots
                     & \int_{\partial K}(\vchi_{N_V}\cdot\bn)\phi_{b,N_b}\, ds\end{bmatrix}.
$$
Notice that $D_K$ is a symmetric matrix.

Once the matrices $D_K$, $Z_K$ and $T_K$ are computed, we can use
(\ref{eq:localwg}) to calculate the weak gradient of basis functions
$\phi_{0,i}$ and $\phi_{b,i}$ on $K$. It is not hard to see that
\begin{equation} \label{eq:wgvectorform}
\underline{(\nabla_d \phi_{0,i})} = -D_K^{-1}Z_K\underline{\ve}_{i}^{N_0},\qquad
\underline{(\nabla_d \phi_{b,i})} = D_K^{-1}T_K \underline{\ve}_{i}^{N_b},
\end{equation}
where $\underline{\ve}_{i}^{N_0}$ and $\underline{\ve}_{i}^{N_b}$
are the standard basis for the Euclidean spaces $\bbr^{N_0}$ and $\bbr^{N_b}$, respectively,
such that its $i$-th entry is $1$ and all other entries are $0$.

Define matrices
\begin{equation} \label{eq:ABCK}
\begin{aligned}
  A_K &= \begin{bmatrix}(\calA \vchi_j,\vchi_i)_K\end{bmatrix}_{i,j},\\
  B_K &= \begin{bmatrix}(\vbeta\cdot\vchi_j,\phi_{0,i})_K\end{bmatrix}_{i,j},\\
  C_K &= \begin{bmatrix}(\gamma\phi_{0,j},\phi_{0,i})_K\end{bmatrix}_{i,j},
\end{aligned}
\end{equation}
where $(\cdot,\cdot)_K$ denote the standard inner-product on
$L^2(K)$ or $[L^2(K)]^d$, as appropriate. Clearly, $A_K$ is an
$N_V\times N_V$ matrix, $B_K$ is an $N_0\times N_V$ matrix, and
$C_K$ is an $N_0\times N_0$ matrix. Then, an elementary matrix
calculation shows that the local stiffness matrix $M_K$ for Equation
(\ref{eq:weakformulation}) can be expressed in a way as specified in
the following lemma. \smallskip
\begin{lemma} \label{lem:localstiffness}
The local stiffness matrix $M_K$ defined in
(\ref{eq:localstiffness}) can be computed by using the following
formula
\begin{equation}\label{localstiffnessmatrixMk}
 \begin{aligned}
   M_{0,0} &= Z_K^t D_K^{-t} A_K D_K^{-1} Z_K - B_K D_K^{-1} Z_K +  C_K,\\
   M_{0,b} &= - Z_K^t D_K^{-t} A_K D_K^{-1} T_K + B_K D_K^{-1} T_K, \\
   M_{b,0} &= - T_K^t D_K^{-t} A_K D_K^{-1} Z_K + T_K^t D_K^{-t} B_K^t,\\
   M_{b,b} &= T_K^t D_K^{-t} A_K D_K^{-1} T_K,\\
 \end{aligned}
\end{equation}
where the superscript $t$ stands for the standard matrix transpose.
\end{lemma}

\medskip

For the Poisson equation $-\Delta u = f$, we clearly have $A_K=D_K$
and $B_K = 0$, $C_K = 0$. Since $D_K$ is symmetric, the local
stiffness matrix becomes
\begin{equation}\label{poisson.01}
M_K = \begin{bmatrix}Z_K^t D_K^{-1}Z_K & -Z_K^t D_K^{-1}T_K \\[2mm] -T_K^t D_K^{-1}Z_K & T_K^t D_K^{-1}T_K\end{bmatrix}.
\end{equation}
In the rest of this section, we shall demonstrate the computation of
the element stiffness matrix $M_K$ with two concrete examples.

%%%%%%%%%%%%%%%%%%%%%%%%%%%%%%%%%%%%%%%%%%%%%%%%%%%%%%%%%%%%%%%%%%%%%%%
\subsection{For the Triangular Element ${\bf (P_0(K),\,P_0(F),\, RT_0(K))}$}
Let $K$ be a triangular element in $\calT_h$. We consider the case
when $j = l = 0$ and $V_r(K)$ being the lowest order Raviart-Thomas
element. In other words, the discrete space $S_h$ consists of
piecewise constants on the triangles, and piecewise constants on the
edges of the mesh. In this case, the discrete gradient is defined by
using the lowest order Raviart-Thomas element on the triangle $K$.
Clearly, we have $N_0 = 1$, $N_b=3$ and $N_V = 3$.

Let $v_i=(x_i, y_i)$, $i=1,2,3$, be the vertices of the triangle $K$
and $e_i$ be the edge opposite to the vertex $v_i$. Denote by
$|e_i|$ the length of edge $e_i$ and $|K|$ the area of the triangle
$K$. We also denote by $\bn_i$ and $\bt_i$ the unit outward normal
and unit tangential vectors on $e_i$, respectively. Here $\bt_i$
should be in the positive (counterclockwise) orientation. If edge
$e_i$ goes from vertex $v_j$ to $v_k$ and $K$ stays on the left when
one travels from $v_j$ to $v_k$, then it is not hard to see that
$$
\bt_i = \begin{bmatrix}t_{i,1}\\t_{i,2}\end{bmatrix} = \frac{1}{|e_i|}\begin{bmatrix}x_{k}-x_j \\ y_{k}-y_j\end{bmatrix},
\qquad\quad
\bn_i = \begin{bmatrix}n_{i,1}\\n_{i,2}\end{bmatrix} = \frac{1}{|e_i|}\begin{bmatrix}y_{k}-y_j \\ -(x_{k}-x_j)\end{bmatrix}.
$$

\subsubsection{Approach I} One may use the following set of basis functions for
the weak discrete functions on $K$:
\begin{equation}\label{phibasis}
\phi_{0,1} = 1,\qquad
\phi_{b,i} = \begin{cases}1 \quad&\textrm{on }e_i\\[1mm]0&\textrm{otherwise}\end{cases}\quad \textrm{ for }i=1,2,3,
\end{equation}
and
\begin{equation} \label{eq:triRTbasis}
\vchi_i = \frac{|e_i|}{2|K|} \begin{bmatrix}x-x_i \\ y-y_i\end{bmatrix},\quad \textrm{ for }i=1,2,3.
\end{equation}
Notice that $\vchi_i$ forms the standard basis for the lowest order Raviart-Thomas element,
for which the degrees of freedom are taken to be the normal component on edges.
Indeed, $\vchi_i$ satisfies
$$
\vchi_i\cdot\vn_j|_{e_j} = \begin{cases}1 \qquad&\textrm{for }i=j,\\0 \qquad&\textrm{for }i\neq j. \end{cases}
$$
It is straight forward to compute that, for the above defined basis functions,
$$
Z_K = \begin{bmatrix}|e_1| \\ |e_2| \\ |e_3|\end{bmatrix},\qquad
T_K = \begin{bmatrix}|e_1|&0&0\\0&|e_2|&0\\0&0&|e_3|\end{bmatrix}.
$$
The computation of $D_K$ is slightly more complicated, but it can
still be done without much difficulty, especially with the help of
symbolic computing tools provided in existing software packages such
as Maple and Mathematica. For simplicity of notation, denote
$$
\begin{aligned}
l_{i} &= |e_i|^2 \qquad &&\textrm{for }1\le i\le 3,\\
l_{ij} &= |e_i|^2+|e_j|^2\qquad &&\textrm{for }1\le i,j\le 3\textrm{ and }i\neq j,\\
l_{123} &= |e_1|^2 + |e_2|^2 + |e_3|^2. &&
\end{aligned}
$$
Then, it can be verified that
\begin{equation} \label{eq:DKRT0}
\begin{aligned}
D_K &= \frac{1}{48|K|}\begin{bmatrix}
  |e_1|^2\left(3l_{23}-l_1\right) & |e_1||e_2|\left(l_{12}-3l_3\right) & |e_1||e_3|\left(l_{13}-3l_2\right) \\[1mm]
  |e_1||e_2|\left(l_{12}-3l_3\right) & |e_2|^2\left(3l_{13}-l_2\right) & |e_2||e_3|\left(l_{23}-3l_1\right) \\[1mm]
  |e_1||e_3|\left(l_{13}-3l_2\right) & |e_2||e_3|\left(l_{23}-3l_1\right) & |e_3|^2\left(3l_{12}-l_3\right)
\end{bmatrix} \\[2mm]
&= \frac{1}{48|K|} T_K \begin{bmatrix}
 3l_{23}-l_1 & l_{12}-3l_3 & l_{13}-3l_2 \\[1mm]
 l_{12}-3l_3 & 3l_{13}-l_2 & l_{23}-3l_1 \\[1mm]
 l_{13}-3l_2 & l_{23}-3l_1 & 3l_{12}-l_3
\end{bmatrix} T_K^t.
\end{aligned}
\end{equation}
We point out that, the value of $D_K$ given as in (\ref{eq:DKRT0})
agrees with the one presented in \cite{Bahriawati}. A verification
of the formula (\ref{eq:DKRT0}) can be carried out by using the
following fact
$$
\begin{aligned}
|K| &= \frac{1}{2} \begin{vmatrix} 1&1&1\\x_1 & x_2 & x_3 \\ y_1 & y_2 & y_3 \end{vmatrix}, \\[2mm]
|e_i|^2 &= (x_j-x_k)^2 + (y_j-y_k)^2,  \qquad i=1,2,3,\; j\neq
k,\;j\textrm{ and }k \textrm{ different from }i.
\end{aligned}
$$
In computer implementation, it is convenient to use a form for the
local matrix that can be expressed by using only edge lengths, as
the one given by (\ref{eq:DKRT0}).

In addition, using symbolic computing tools and the law of sines and
cosines, we can write $D_K^{-1}$ as follows:
$$
D_K^{-1} = T_K^{-t}\left(\frac{16|K|}{l_{123}}\begin{bmatrix}1&1&1\\1&1&1\\1&1&1\end{bmatrix} +\frac{1}{2|K|} \begin{bmatrix}2 l_1 &l_3-l_{12}&l_2-l_{13} \\[1mm]    l_3-l_{12}&2l_2&l_1-l_{23} \\[1mm]    l_2-l_{13}&l_1-l_{23}&2l_3 \\[1mm]  \end{bmatrix} \right) T_K^{-1}.
$$

Thus, to compute the local stiffness matrix $M_K$, it suffices to
calculate $A_K$, $B_K$ and $C_K$ as given in (\ref{eq:ABCK}), and
then apply Lemma \ref{lem:localstiffness}. Notice that these three
matrices depend on the coefficients $\calA$, $\vbeta$ and $\gamma$,
and quadrature rules may be employed in the calculation. However,
for the simple case of the Poisson equation $-\Delta u = f$, we see
from (\ref{poisson.01}) that
$$
\begin{aligned}
M_{00} &= \begin{bmatrix}\frac{144|K|}{l_{123}}\end{bmatrix},\qquad
M_{0b}=M_{b0}^t = \begin{bmatrix}\frac{-48|K|}{l_{123}}& \frac{-48|K|}{l_{123}} & \frac{-48|K|}{l_{123}}\end{bmatrix},\\[2mm]
M_{bb} &= \frac{16|K|}{l_{123}}\begin{bmatrix}1&1&1\\[1mm] 1&1&1\\[1mm] 1&1&1\end{bmatrix}
 + \frac{1}{2|K|}\begin{bmatrix}2 l_1 &l_3-l_{12}&l_2-l_{13} \\[1mm]
   l_3-l_{12}&2l_2&l_1-l_{23} \\[1mm]
   l_2-l_{13}&l_1-l_{23}&2l_3 \\[1mm]
 \end{bmatrix}.
\end{aligned}
$$

\subsubsection{Approach II} We would like to present another approach for
computing the local stiffness matrix $M_K$ in the rest of this
subsection. Observe that a set of basis functions for the local
space $V_r(K)$ can be chosen as follows
\begin{equation}\label{newbasis}
\vchi_1 = \begin{bmatrix}1\\[1mm] 0\end{bmatrix},\qquad
\vchi_2 = \begin{bmatrix}0\\[1mm] 1\end{bmatrix},\qquad
\vchi_3 = \begin{bmatrix}x-\bar{x}\\[1mm] y-\bar{y}\end{bmatrix},
\end{equation}
where $(\bar{x}=(x_1+x_2+x_3)/3, \bar{y}=(y_1+y_2+y_3)/3)$ is the
coordinate of the barycenter of $K$.  Note that both components of
$\vchi_3$ have mean value zero on $K$. For the weak discrete
function on $K$, we use the same set of basis functions as given in
(\ref{phibasis}). It is not hard to see that
$$
D_K = |K|\begin{bmatrix}1&0&0\\[1mm]0&1&0\\[1mm]0&0&\frac{l_{123}}{36}\end{bmatrix},\qquad
Z_K = \begin{bmatrix}0\\[1mm]0\\[1mm]2|K|\end{bmatrix},\qquad
$$
and
$$
T_K = \begin{bmatrix}y_3-y_2 & y_1-y_3 & y_2-y_1 \\[1mm]
                     x_2-x_3 & x_3-x_1 & x_1-x_2 \\[1mm]
                     \frac{2|K|}{3} & \frac{2|K|}{3} & \frac{2|K|}{3} \end{bmatrix}.
$$
Next, we use the formula (\ref{eq:ABCK}) to calculate the matrices
$A_K$, $B_K$ and $C_K$ for the new basis (\ref{newbasis}). Finally,
we calculate the local stiffness matrix $M_K$ by using the formula
provided in Lemma \ref{lem:localstiffness}.

Since the set of basis functions for the weak discrete space are the
same in Approaches I and II, the resulting local stiffness matrix
$M_K$ would remain unchanged from Approaches I and II. The set of
basis functions (\ref{newbasis}) is advantageous over the set
(\ref{eq:triRTbasis}) in that the matrix $D_K$ is a diagonal one
whose inverse in trivial to compute.

%%%%%%%%%%%%%%%%%%%%%%%%%%%%%%%%%%%%%%%%%%%%%%%%%%%%%%%%%%%%%%%%%%%%%%%
\subsection{For the Cubic Element ${\bf (Q_0(K),\,Q_0(F),\, RT_0(K))}$}
Let $K=[0,a]\times [0,b]\times [0,c]$ be a rectangular box where $a,
b, c$ are positive real numbers. We consider the three-dimensional
cubic element, for which the discrete space $S_h$ consists of
piecewise constants on $K_0$ and piecewise constants on the faces of
$K$. The space for the discrete gradient is the lowest order
Raviart-Thomas element on $K$. We clearly have $N_0 = 1$, $N_b=6$
and $N_V = 6$.

Denote the six faces $F_i$, $i=1,\ldots,6$ by
$$
\begin{aligned}
  {F}_1 \;:\;x &= 0,\qquad  &{F}_2 \;:\;x &= a, \\
  {F}_3 \;:\;y &= 0,\qquad  &{F}_4 \;:\;y &= b, \\
  {F}_5 \;:\;z &= 0,\qquad  &{F}_6 \;:\;z &= c. \\
\end{aligned}
$$
Note that the volume of $K$ is given by $|K|=abc$ and the normal
direction to each face is given by
$$
\bn_1 = \begin{bmatrix}-1\\0\\0\end{bmatrix},\;
\bn_2 = \begin{bmatrix}1\\0\\0\end{bmatrix},\;
\bn_3 = \begin{bmatrix}0\\-1\\0\end{bmatrix},\;
\bn_4 = \begin{bmatrix}0\\1\\0\end{bmatrix},\;
\bn_5 = \begin{bmatrix}0\\0\\-1\end{bmatrix},\;
\bn_6 = \begin{bmatrix}0\\0\\1\end{bmatrix}.
$$
We adopt the following set of basis functions for the weak discrete
space on $K$
$$
\phi_{0,1} = 1,\qquad
\phi_{b,i} = \begin{cases}1 \quad&\textrm{on }F_i\\[1mm]0&\textrm{otherwise}\end{cases}\quad \textrm{ for }i=1,\ldots,6,
$$
and
$$
\vchi_1 = \begin{bmatrix}\frac{x}{a}-1\\0\\0\end{bmatrix},\,
\vchi_2 = \begin{bmatrix}\frac{x}{a}\\0\\0\end{bmatrix},\,
\vchi_3 = \begin{bmatrix}0\\\frac{y}{b}-1\\0\end{bmatrix},\,
\vchi_4 = \begin{bmatrix}0\\\frac{y}{b}\\0\end{bmatrix},\,
\vchi_5 = \begin{bmatrix}0\\0\\\frac{z}{c}-1\end{bmatrix},\,
\vchi_6 = \begin{bmatrix}0\\0\\\frac{z}{c}\end{bmatrix}.
$$
Clearly, each $\vchi_i$ satisfies
$$
\vchi_i\cdot\vn_j|_{F_j} = \begin{cases}1 \qquad&\textrm{for }i=j,\\0 \qquad&\textrm{for }i\neq j. \end{cases}
$$

It is not hard to compute that
$$
D_K = \frac{|K|}{6}\begin{bmatrix}2&-1&0&0&0&0\\ -1&2&0&0&0&0 \\ 0&0&2&-1&0&0 \\  0&0&-1&2&0&0\\ 0&0&0&0&2&-1\\ 0&0&0&0&-1&2\end{bmatrix},\qquad
D_K^{-1} = \frac{2}{|K|}\begin{bmatrix}2&1&0&0&0&0\\ 1&2&0&0&0&0 \\ 0&0&2&1&0&0 \\  0&0&1&2&0&0\\ 0&0&0&0&2&1\\ 0&0&0&0&1&2 \end{bmatrix},
$$
and
$$
Z_K = \begin{bmatrix}bc\\bc\\ac\\ac\\ab\\ab\end{bmatrix},\qquad
T_K = \begin{bmatrix}bc&0&0&0&0&0\\ 0&bc&0&0&0&0 \\ 0&0&ac&0&0&0\\ 0&0&0&ac&0&0\\ 0&0&0&0&ab&0\\ 0&0&0&0&0&ab\end{bmatrix}
$$
Then, the local stiffness matrix $M_K$ can be computed using the
formula presented in Lemma \ref{lem:localstiffness}.

%%%%%%%%%%%%%%%%%%%%%%%%%%%%%%%%%%%%%%%%%%%%%%%%%%%%%%%%%%%%%%%%%%%%%%%%%%%%%%%%%
%%%%%%%%%%%%%%%%%%%%%%%%%%%%%%%%%%%%%%%%%%%%%%%%%%%%%%%%%%%%%%%%%%%%%%%%%%%%%%%%%
%%%%%%%%%%%%%%%%%%%%%%%%%%%%%%%%%%%%%%%%%%%%%%%%%%%%%%%%%%%%%%%%%%%%%%%%%%%%%%%%%
\section{Numerical Experiments}

\smallskip
In this section, we shall report some numerical results for the weak
Galerkin finite element method on a variety of testing problems,
with different mesh and finite elements. To this end, let
$u_h=\{u_0,u_b\}$ and $u$ be the solution to the weak Galerkin
equation (\ref{eq:weakformulation}) and the original equation
(\ref{eq:ellipticproblem}), respectively. Define the error by $e_h =
u_h-Q_h u = \{e_0,\,e_b\}$ where $Q_h u$ is the $L^2$ projection of
$u$ onto appropriately defined spaces. Let us introduce the
following norms:
$$
\begin{aligned}
  H^1\textrm{ semi-norm:}&\qquad \|\nabla_d e_h\| = \left(\sum_{K\in\calT_h} \int_K |\nabla_d e_h|^2\,dx\right)^{1/2} ,\\[2mm]
  \textrm{Element-based $L^2$ norm}:&\qquad \|e_0\| = \left(\sum_{K\in\calT_h} \int_K |e_0|^2\,dx\right)^{1/2},\\[2mm]
  \textrm{Edge/Face-based $L^2$ norm}:&\qquad \|e_b\| = \left(\sum_{F\in\calF_h}
  h_K\int_F |e_b|^2\,ds\right)^{1/2},
\end{aligned}
$$
where in the definition of $\|e_b\|$, $h_K$ stands for the size of
the element $K$ that takes $F$ as an edge/face. We shall also
compute the error in the following metrics
$$
\begin{aligned}
\|\nabla_d u_h - \nabla u\| &= \left(\sum_{K\in\calT_h} \int_K |\nabla_d u_h -
\nabla u|^2\,dx\right)^{1/2},\\[2mm]
\|u_h - u\| &= \left(\sum_{K\in\calT_h} \int_K |u_0 -
u|^2\,dx\right)^{1/2},\\[2mm]
\|e_0\|_{\infty} &= \sup_{\scriptsize \begin{matrix}x\in K_0\\
K\in\calT_h\end{matrix}} |e_0(x)|.
\end{aligned}
$$
Here the maximum norm $\|e_0\|_{\infty}$ is computed over all
Gaussian points, and all other integrals are calculated with a
Gaussian quadrature rule that is of high order of accuracy so that
the error from the numerical integration can be virtually ignored.

%%%%%%%%%%%%%%%%%%%%%%%%%%%%%%%%%%%%%%%%%%%%%%%%%%%%%%%%%%%%%%%%%%%%%%%%%%%%%%%%%
\subsection{Case 1: Model Problems with Various Boundary Conditions} First, we consider
the Laplace equation with nonhomogeneous Dirichlet boundary
condition:
\begin{equation} \label{eq:inhomogeneous}
%\begin{cases}
%   -\Delta u = f\qquad &\textrm{in } \Omega, \\
   u = g \qquad \textrm{on } \partial\Omega.
%\end{cases}
\end{equation}
We introduce a discrete Dirichlet boundary data $g_h$, which is
either the usual nodal value interpolation, or the $L^2$ projection
of $u=g$ on the boundary. Let $\Gamma\subset\partial\Omega$ and
define
$$
\begin{aligned}
S_{g_h,\Gamma,h} = & \{v:\:  v|_{K_0}\in P_j(K_0) \textrm{ or }Q_j(K_0) \textrm{ for all } K\in \calT_h,\\
 &\qquad  v|_F\in P_l(F) \textrm{ or }Q_l(F)  \textrm{ for all } F\in \calF_h, \\
 &\qquad  v = g_h\textrm{ on }\calF_h\cap\Gamma\}.
\end{aligned}
$$
When $\Gamma = \partial\Omega$, we simply denote
$S_{g_h,\partial\Omega,h}$ by $S_{g_h,h}$. The discrete Galerkin
formulation for the nonhomogeneous Dirichlet boundary value problem
can be written as: find $u_h\in S_{g_h,h}$ such that for all $v_h\in
S_{h}^0$,
$$
(\calA\nabla_d u_h, \nabla_d v_h)  + (\vbeta \cdot \nabla_d u_h, v_{0})
+ (\gamma u_{0},\,v_{0}) = (f,\, v_0).
$$

We would like to see how the weak Galerkin approximation might be
affected when the boundary data $u=g$ is approximated with different
schemes (nodal interpolation verses $L^2$ projection). To this end,
we use a two dimensional test problem with domain $\Omega =
(0,1)\times(0,1)$ and exact solution given by $u=\sin(2\pi x +
\pi/2)\sin(2\pi y + \pi/2)$. A uniform triangular mesh and the
element $(P_0(K),\,P_0(F),\, RT_0(K))$ is used in the weak Galerkin
discretization. The results are reported in Table
\ref{tab:inhomogeneous1} and Table \ref{tab:inhomogeneous2}. It can
be seen that both approximations of the Dirichlet boundary data give
optimal order of convergence for the weak Galerkin method, while the
$L^2$ projection method yields a slightly smaller error in $\|e_0\|$
and $\|e_b\|$.
\medskip

Next, we consider a mixed boundary condition:
$$
\begin{cases}
  u = g^D \qquad &\textrm{on } \Gamma_D, \\
  (\calA\nabla u)\cdot\vn + \alpha u = g^R \qquad &\textrm{on } \Gamma_R,
\end{cases}
$$
where $g^D$ is the Dirichlet boundary data, $g^R$ is the Robin type boundary data,
$\alpha\ge 0$, and $\Gamma_D\cap \Gamma_R=\emptyset$, $\Gamma_D\cup \Gamma_R=\partial\Omega$.
When $\alpha = 0$, the Robin type boundary condition becomes the Neumann type boundary condition.

For the mixed boundary condition, it is not hard to see that the weak formulation can be written as:
find $u_h\in S_{g^D_h,\Gamma_D,h}$ such that for all $v_h\in S_{0,\Gamma_D,h}$,
$$
(\calA\nabla_d u_h, \nabla_d v_h)+\langle\alpha u_b,
v_b\rangle_{\Gamma_R} + (\vbeta \cdot \nabla_d u_h, v_{0}) + (\gamma
u_{0},\,v_{0}) = (f, v_{0}) + \langle g^R, v_b\rangle_{\Gamma_R},
$$
where $\langle \cdot,\cdot\rangle_{\Gamma_R}$ denotes the $L^2$
inner-product on $\Gamma_R$. We tested a two-dimensional problem
with $\calA$ to be an identity matrix and $\Omega = (0,1)^2$ with a
uniform triangular mesh. The exact solution is chosen to be
$u=\sin(\pi y)e^{-x}$. This function satisfies
$$
\nabla u\cdot\vn + u = 0
$$
on the boundary segment $x=1$. We use the Dirichlet boundary
condition on all other boundary segments. The element
$(P_0(K),\,P_0(F),\, RT_0(K))$ is used in the discretization. For
the Dirichlet boundary data, the $L^2$ projection is used to
approximate the boundary data $g^D_h$. The results are reported in
Table \ref{tab:inhomogeneous3}. It shows optimal rates of
convergence in all norms for the weak Galerkin approximation with
mixed boundary conditions.

\begin{table}
  \caption{Case 1. Numerical results with Dirichlet data being approximated by the
  usual nodal point interpolation.}
  \label{tab:inhomogeneous1}
\begin{tabular}{||c||cccccc||}
    \hline\hline
   $h$ & $\|\nabla_d e_h\|$ & $\|e_0\|$ & $\|e_b\|$ & $\|\nabla_d u_h - \nabla u\|$ & $\|u_0-u\|$ & $\|e_0\|_{\infty}$ \\
    \hline\hline
   1/8   & 7.14e-01 &  2.16e-02 & 4.05e-02  & 1.01e+0  & 1.30e-01 & 4.43e-02  \\ \hline
   1/16  & 3.56e-01 &  5.61e-03 & 1.01e-02  & 5.04e-01 & 6.53e-02 & 1.12e-02  \\ \hline
   1/32  & 1.78e-01 &  1.41e-03 & 2.53e-03  & 2.51e-01 & 3.27e-02 & 2.86e-03  \\ \hline
   1/64  & 8.90e-02 &  3.55e-04 & 6.32e-04  & 1.25e-01 & 1.63e-02 & 7.15e-04  \\ \hline
   1/128 & 4.45e-02 &  8.88e-05 & 1.57e-04  & 6.29e-02 & 8.18e-03 & 1.79e-04  \\ \hline\hline
   $\begin{matrix}O(h^r)\\r=\end{matrix}$ & 1.0012  &  1.9837  &  2.0014 & 1.0024 &  0.9984  &  1.9879 \\ \hline\hline
   \end{tabular}
\end{table}

\begin{table}
  \caption{Case 1. Numerical results with Dirichlet data being approximated by
  $L^2$ projection.}
  \label{tab:inhomogeneous2}
\begin{tabular}{||c||cccccc||}
    \hline\hline
   $h$ & $\|\nabla_d e_h\|$ & $\|e_0\|$ & $\|e_b\|$ & $\|\nabla_d u_h - \nabla u\|$ & $\|u_0-u\|$ & $\|e_0\|_{\infty}$ \\
    \hline\hline
   1/8   & 7.10e-01 &  1.75e-02 & 3.08e-02  & 1.01e+0  & 1.29e-01 &  3.68e-02 \\ \hline
   1/16  & 3.55e-01 &  4.59e-03 & 7.69e-03  & 5.04e-01 & 6.52e-02 &  9.54e-03 \\ \hline
   1/32  & 1.78e-01 &  1.16e-03 & 1.92e-03  & 2.51e-01 & 3.27e-02 &  2.39e-03 \\ \hline
   1/64  & 8.90e-02 &  2.90e-04 & 4.81e-04  & 1.25e-01 & 1.63e-02 &  6.01e-04 \\ \hline
   1/128 & 4.45e-02 &  7.27e-05 & 1.20e-04  & 6.29e-02 & 8.18e-03 &  1.50e-04 \\ \hline\hline
   $\begin{matrix}O(h^r)\\r=\end{matrix}$ & 0.9993  &  1.9808  &  1.9999  & 1.0015 & 0.9968  &  1.9861 \\ \hline\hline
   \end{tabular}
\end{table}

\begin{table}
  \caption{Case 1. Numerical results for a test problem with mixed boundary conditions,
  where a Robin type boundary condition is imposed on part of the boundary.}
  \label{tab:inhomogeneous3}
\begin{tabular}{||c||cccccc||}
    \hline\hline
   $h$ & $\|\nabla_d e_h\|$ & $\|e_0\|$ & $\|e_b\|$ & $\|\nabla_d u_h - \nabla u\|$ & $\|u_0-u\|$ & $\|e_0\|_{\infty}$ \\
    \hline\hline
   1/8   & 1.55e-01 &  3.18e-03 & 1.14e-02  &  1.95e-01 &  4.51e-02 & 1.12e-02  \\ \hline
   1/16  & 7.87e-02 &  8.20e-04 & 2.90e-03  &  9.82e-02 &  2.25e-02 & 3.18e-03  \\ \hline
   1/32  & 3.94e-02 &  2.06e-04 & 7.29e-04  &  4.92e-02 &  1.12e-02 & 8.40e-04  \\ \hline
   1/64  & 1.97e-02 &  5.17e-05 & 1.82e-04  &  2.46e-02 &  5.64e-03 & 2.15e-04  \\ \hline
   1/128 & 9.87e-03 &  1.29e-05 & 4.56e-05  &  1.23e-02 &  2.82e-03 & 5.46e-05  \\ \hline\hline
   $\begin{matrix}O(h^r)\\r=\end{matrix}$ & 0.9958 &   1.9876  &  1.9926 &   0.9971  &  1.0001  &  1.9262 \\ \hline\hline
   \end{tabular}
\end{table}

%%%%%%%%%%%%%%%%%%%%%%%%%%%%%%%%%%%%%%%%%%%%%%%%%%%%%%%%%%%%%%%%%%%%%%%%%%%%%%%%%
\subsection{Case 2: A Model Problem with Degenerate Diffusion} We consider a test problem
where the diffusive coefficient $\calA$ is singular at some points
of the domain. Note that in this case, the usual mixed finite
element method may not be applicable due to the degeneracy of the
coefficient. But the primary variable based formulations, including
the weak Galerkin method, can still be employed for a numerical
approximation.

More precisely, we consider the following two-dimensional problem
\begin{eqnarray*}
    -\nabla\cdot(xy\,\nabla u) &=& f\qquad \textrm{in } \Omega, \\
    u &=& 0 \qquad \textrm{on } \partial\Omega,
  \end{eqnarray*}
where $\Omega=(0,1)^2$. Notice that the diffusive coefficient
$\calA=xy$ vanishes at the origin. We set the exact solution to be
$u=x(1-x)y(1-y)$. The configuration for the finite element
partitions is the same as in test Case 1. We tested the weak
Galerkin method on this problem, and the results are presented in
Table \ref{tab:variableA} and Figure \ref{fig:variableA}.

Since the diffusive coefficient $\calA$ is not uniformly positive
definite on $\Omega$, we have no anticipation that the weak Galerkin
approximation has any optimal rate of convergence, though the exact
solution is smooth. It should be pointed out that the usual
Lax-Milgram theorem is not applicable to such problems in order to
have a result on the solution existence and uniqueness. However, one
can prove that the discrete problem always has a unique solution
when Gaussian quadratures are used in the numerical integration.
Interestingly, the numerical experiments show that the weak Galerkin
method converges with a rate of approximately $O(h^{0.5})$ in
$\|\nabla_d e_h\|$, $O(h^{1.25})$ in $\|e_0\|$ and $\|e_b\|$. It is
left for future research to explore a theoretical foundation of the
observed convergence behavior.

\begin{table}
  \caption{Case 2. Numerical results for a test problem with degenerate diffusion $\calA$ in the
  domain.}
  \label{tab:variableA}
  \begin{tabular}{||c||cccccc||}
    \hline\hline
   $h$ & $\|\nabla_d e_h\|$ & $\|e_0\|$ & $\|e_b\|$ & $\|\nabla_d u_h - \nabla u\|$ & $\|u_0-u\|$ & $\|e_0\|_{\infty}$ \\
    \hline\hline
   1/8   & 5.61e-02 &  3.32e-03 & 6.60e-03  &  5.75e-02 &  5.48e-03 & 1.27e-02  \\ \hline
   1/16  & 4.03e-02 &  1.38e-03 & 2.81e-03  &  4.09e-02 &  2.59e-03 & 4.90e-03  \\ \hline
   1/32  & 2.95e-02 &  5.68e-04 & 1.16e-03  &  2.96e-02 &  1.23e-03 & 2.21e-03  \\ \hline
   1/64  & 2.15e-02 &  2.35e-04 & 4.83e-04  &  2.15e-02 &  5.97e-04 & 1.16e-03  \\ \hline
   1/128 & 1.55e-02 &  9.93e-05 & 2.02e-04  &  1.55e-02 &  2.91e-04 & 5.99e-04  \\ \hline\hline
   $\begin{matrix}O(h^r)\\r=\end{matrix}$ & 0.4614 &   1.2687  &  1.2594 &   0.4697 &   1.0579 &  1.0912  \\ \hline\hline
   \end{tabular}
\end{table}

\begin{figure}
  \begin{center}
  \caption{Case 2. Convergence rate of $\|\nabla_d e_h\|$, $\|e_0\|$ and
  $\|e_b\|$ for the case of degenerate diffusions.}
  \label{fig:variableA}
  \includegraphics[width=8cm]{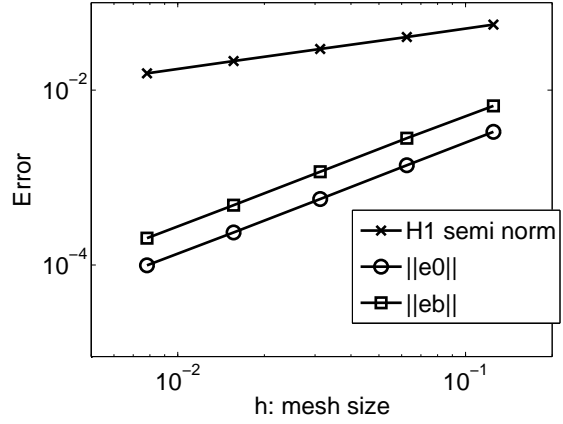}
  \end{center}
\end{figure}

%%%%%%%%%%%%%%%%%%%%%%%%%%%%%%%%%%%%%%%%%%%%%%%%%%%%%%%%%%%%%%%%%%%%%%%%%%%%%%%%%
\subsection{Case 3: A Model Problem on a Domain with Corner Singularity} We consider
the Laplace equation on a two-dimensional domain for which the exact
solution possesses a corner singularity. For simplicity, we take
$\Omega=(0,1)^2 $ and let the exact solution be given by
\begin{equation}
  u(x,y) = x(1-x) y(1-y) r^{-2+\gamma},
\end{equation}
where $ r = \sqrt{x^2+y^2} $ and $ \gamma \in (0,1] $ is a constant.
Clearly, we have
$$
  u \in H^1_0(\Omega) \cap H^{1+\gamma-\varepsilon}(\Omega)
  \quad \textrm{ and }\quad
  u \notin H^{1+\gamma}(\Omega),
$$
where $ \varepsilon $ is any small, but positive number. Again, a
uniform triangular mesh and the element $(P_0(K),\,P_0(F),\,
RT_0(K))$ are used in the numerical discretization. Note that the
weak Galerkin for this problem is exactly the same as the standard
mixed finite element method.

This model problem was numerically tested with $ \gamma=0.5 $ and $
\gamma=0.25 $. The convergence rates are reported in Table
\ref{tab:singularity1} and Table \ref{tab:singularity2}. Notice that
$\|\nabla_d e_h\|$ and $\|e_0\|$ behaves in a way as predicted by
theory (\ref{eq:errorestimation}); i.e., they converge with rates
given by $O(h^{\gamma})$ and $O(h^{1+\gamma})$, respectively. The
result also shows that the approximation on the element edge/face
has a rate of convergence $O(h^{1+\gamma})$.

\begin{table}
  \caption{Case 3. Convergence rates for a problem with corner singularity ($\gamma = 0.5$).}
  \label{tab:singularity1}
\begin{tabular}{||c||cccccc||}
    \hline\hline
   $h$ & $\|\nabla_d e_h\|$ & $\|e_0\|$ & $\|e_b\|$ & $\|\nabla_d u_h - \nabla u\|$ & $\|u_0-u\|$ & $\|e_0\|_{\infty}$ \\
    \hline\hline
   1/8   & 1.88e-01 &  6.40e-03 & 1.47e-02  &  2.54e-01 &  1.49e-02 &  4.30e-02 \\ \hline
   1/16  & 1.36e-01 &  2.20e-03 & 5.28e-03  &  1.84e-01 &  7.66e-03 &  3.01e-02 \\ \hline
   1/32  & 9.74e-02 &  7.62e-04 & 1.86e-03  &  1.32e-01 &  3.89e-03 &  2.12e-02 \\ \hline
   1/64  & 6.93e-02 &  2.65e-04 & 6.57e-04  &  9.42e-02 &  1.96e-03 &  1.49e-02 \\ \hline
   1/128 & 4.92e-02 &  9.33e-05 & 2.32e-04  &  6.69e-02 &  9.88e-04 &  1.05e-02 \\ \hline\hline
   $\begin{matrix}O(h^r)\\r=\end{matrix}$ & 0.4852 &   1.5251  &  1.4992  &  0.4827 &   0.9805 &   0.5066 \\ \hline\hline
   \end{tabular}
\end{table}

\begin{table}
  \caption{Case 3. Convergence rates for a problem with corner singularity ($\gamma = 0.25$).}
  \label{tab:singularity2}
\begin{tabular}{||c||cccccc||}
    \hline\hline
   $h$ & $\|\nabla_d e_h\|$ & $\|e_0\|$ & $\|e_b\|$ & $\|\nabla_d u_h - \nabla u\|$ & $\|u_0-u\|$ & $\|e_0\|_{\infty}$ \\
    \hline\hline
   1/8   & 4.93e-01 &  1.69e-02 & 3.58e-02  &  6.65e-01 &  2.56e-02 &  1.25e-01 \\ \hline
   1/16  & 4.18e-01 &  7.07e-03 & 1.52e-02  &  5.66e-01 &  1.31e-02 &  1.05e-01 \\ \hline
   1/32  & 3.53e-01 &  2.94e-03 & 6.39e-03  &  4.79e-01 &  6.72e-03 &  8.85e-02 \\ \hline
   1/64  & 2.98e-01 &  1.22e-03 & 2.68e-03  &  4.04e-01 &  3.42e-03 &  7.44e-02 \\ \hline
   1/128 & 2.51e-01 &  5.14e-04 & 1.12e-03  &  3.40e-01 &  1.73e-03 &  6.25e-02 \\ \hline\hline
   $\begin{matrix}O(h^r)\\r=\end{matrix}$ & 0.2437  &  1.2613  &  1.2489  &  0.2417  &  0.9717  &  0.2505 \\ \hline\hline
   \end{tabular}
\end{table}

%%%%%%%%%%%%%%%%%%%%%%%%%%%%%%%%%%%%%%%%%%%%%%%%%%%%%%%%%%%%%%%%%%%%%%%%%%%%%%%%%
\subsection{Case 4: A Model Problem with Intersecting Interfaces} This test problem is
taken from \cite{Kell_ApplAnal_1975}, which has also been tested by
other researchers
\cite{LiuMuYe_PrepJCAM_2011,MorinNocheSie_SIREV_2002}. In two
dimension, consider $ \Omega = (-1,1)^2 $ and the following problem
$$
  -\nabla \cdot (\calA \nabla u) = 0,
$$
where $\calA = K_1 \mathbf{I}_2$ in the first and third quadrants,
and $K_2 \mathbf{I}_2 $ in the second and forth quadrants. Here
$\mathbf{I}_2$ is the $2\times 2$ identity matrix and $K_1$, $K_2$
are two positive numbers. Consider an exact solution which takes the
following form in polar coordinates:
$$
u(x,y) = r^{\gamma} \mu(\theta),
$$
where $ \gamma \in (0,1] $ and
\begin{equation}
  \displaystyle
  \mu(\theta) =
  \left\{
    \begin{array}{ll}
      \cos((\pi/2-\sigma) \gamma) \cos((\theta-\pi/2+\rho)\gamma),
        & \mbox{if} \; 0 \le \theta \le \pi/2,
      \\
      \cos(\rho \gamma) \cos((\theta-\pi+\sigma)\gamma),
        & \mbox{if} \; \pi/2 \le \theta \le \pi,
      \\
      \cos(\sigma \gamma) \cos((\theta-\pi-\rho)\gamma),
        & \mbox{if} \; \pi \le \theta \le 3\pi/2,
      \\
      \cos((\pi/2-\rho) \gamma) \cos((\theta-3\pi/2-\sigma)\gamma),
        & \mbox{if} \; 3\pi/2 \le \theta \le 2\pi.
    \end{array}
  \right.
\end{equation}
The parameters $ \gamma, \rho, \sigma $
satisfy the following nonlinear relations
\begin{equation}
  \begin{array}{l}
    R: = K_1/K_2 = -\tan((\pi/2-\sigma) \gamma) \cot(\rho \gamma),
    \\
    1/R = -\tan(\rho \gamma) \cot(\sigma \rho),
    \\
    R = -\tan(\rho \gamma) \cot((\pi/2-\rho) \gamma),
    \\
    \max \{ 0,\pi \gamma -\pi \} < 2 \gamma \rho < \min \{ \pi\gamma,\pi \},
    \\
    \max \{0,\pi-\pi \gamma \} < -2\gamma \sigma < \min \{ \pi,2\pi-\pi\gamma \}.
  \end{array}
\end{equation}
The solution $ u(r,\theta) $ is known to be in
$H^{1+\gamma-\varepsilon}(\Omega) $ for any $ \varepsilon>0 $, and
has a singularity near the origin $(0,0)$.

\smallskip
One choice for the coefficients is to take $ \gamma = 0.1 $, $ R
\approx 161.4476387975881 $, $ \rho \approx \pi/4 $, $ \sigma
\approx -14.92256510455152$. We numerically solve this problem by
using the weak Galerkin method with element $(P_0(K),\,P_0(F),\,
RT_0(K))$ on triangular meshes. It turns out that uniform triangular
meshes are not good enough to handle the singularity in this
problem. Indeed, we use a locally refined initial mesh, as shown in
Figure \ref{FigExIntscIntfcInitTrigMesh}, which consists of 268
triangles. This mesh is then uniformly refined, by dividing each
triangle into 4 subtriangles, to get a sequence of nested meshes.
Although this can not be compared with an adaptive mesh refinement
process, it does improve the accuracy of the numerical
approximation, as shown in our numerical results reported in Table
\ref{tab:interface}. Since the mesh is not quasi-uniform, we do not
expect that the theoretic error estimation
(\ref{eq:errorestimation}) apply for this problem. An interesting
observation of Table \ref{tab:interface} is that, the norm
$\|u_0-u\|$ appears to converge in a much faster rate than $\|e_0\|
= \|u_0-Q_0 u\|$, while the opposite has usually been observed for
other test cases. We believe that this is due to the use of a
locally refined initial mesh in our testing process. When the actual
value of $\|u_0-u\|$ reduces to the same level as the value of
$\|e_0\|$, its convergence rate slows down to the same as $\|e_0\|$.
Readers are also encouraged to derive their own conclusions from
these numerical experiments.

We also observe that, when the initial mesh gets more refined near
the origin, the convergence rates increase slightly. In Table
\ref{tab:interface2}, this trend is clearly shown. For each initial
mesh, it is refined four times to get five levels of nested meshes.
The convergence rates are computed based on these five nested
meshes. The initial meshes are generated by refining only those
triangles near the origin. Two examples of initial meshes are shown
in Figure \ref{FigExIntscIntfcInitTrigMesh}.

\begin{figure}[!h]
\centering \caption{Case 4. The initial triangular mesh for the
intersecting interface problem, with 268 (left) and 300 (right)
triangles.} \label{FigExIntscIntfcInitTrigMesh}
\includegraphics[width=5cm]{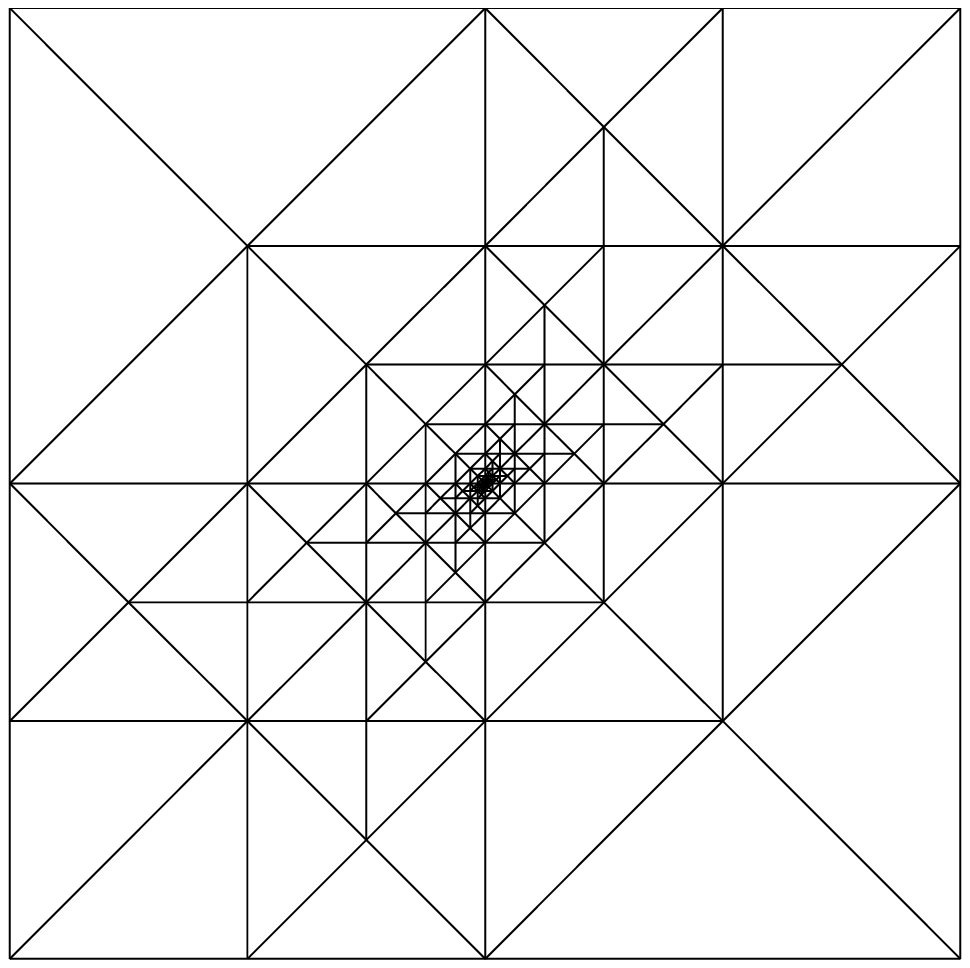}
\includegraphics[width=5cm]{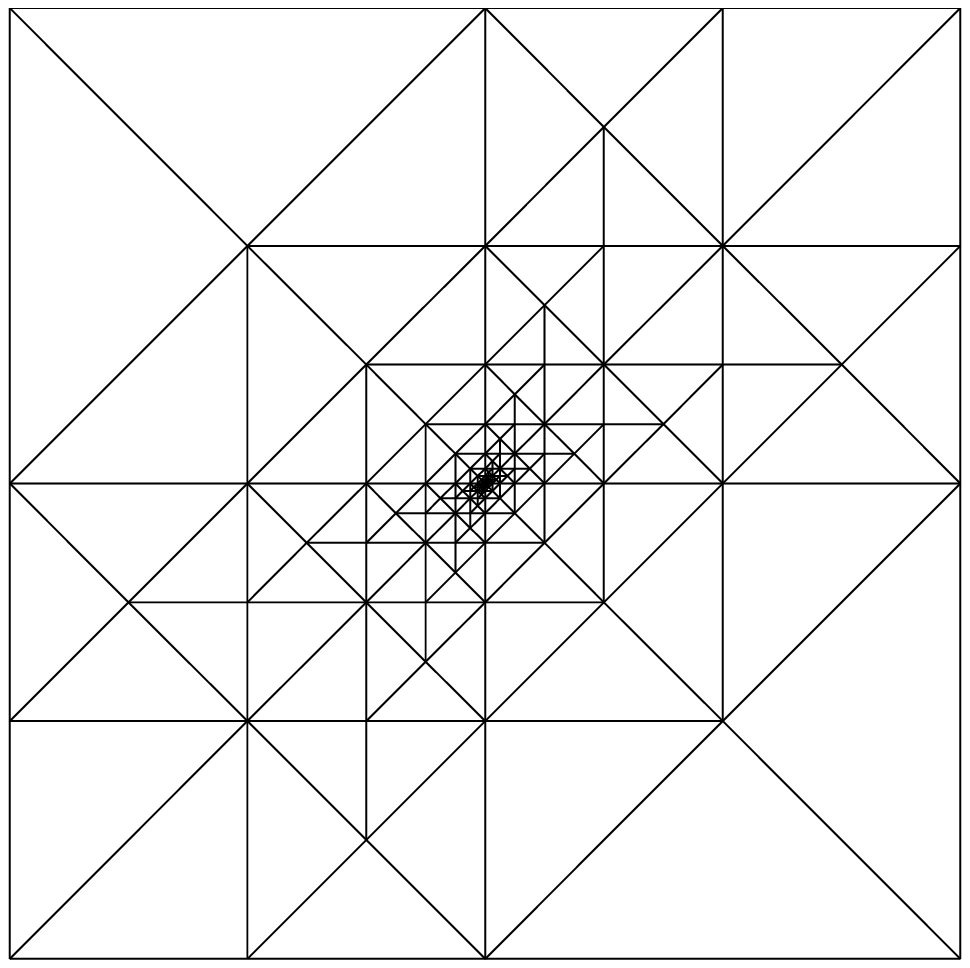}
\end{figure}

\begin{table}
  \caption{Case 4. Convergence rate for the intersecting interface problem with an
  initial mesh containing 268 triangles.}
  \label{tab:interface}
\begin{tabular}{||c||cccccc||}
    \hline\hline
   level & $\|\nabla_d e_h\|$ & $\|e_0\|$ & $\|e_b\|$ & $\|\nabla_d u_h - \nabla u\|$ & $\|u_0-u\|$ & $\|e_0\|_{\infty}$ \\
    \hline\hline
   0 & 1.07e-01 &  3.97e-03 & 9.95e-03  &  1.47e-01 &  2.60e-02 & 1.97e-02  \\ \hline
   1 & 9.76e-02 &  2.92e-03 & 6.44e-03  &  1.26e-01 &  1.33e-02 & 1.94e-02  \\ \hline
   2 & 9.30e-02 &  2.51e-03 & 5.11e-03  &  1.16e-01 &  7.01e-03 & 1.91e-02  \\ \hline
   3 & 9.12e-02 &  2.21e-03 & 4.44e-03  &  1.11e-01 &  3.95e-03 & 1.88e-02  \\ \hline
   4 & 8.98e-02 &  1.95e-03 & 3.91e-03  &  1.07e-01 &  2.55e-03 & 1.84e-02  \\ \hline\hline
   $\begin{matrix}O(h^r)\\r=\end{matrix}$ & 0.0604  &  0.2446 &  0.3229 &   0.1084 &   0.8461 &   0.0239 \\ \hline\hline
   \end{tabular}
\end{table}

\begin{table}
  \caption{Case 4. Convergence rate for the intersecting interface problem with different
  initial meshes, where the first column indicates the total number of triangles in the
  initial mesh.}
  \label{tab:interface2}
\begin{tabular}{||c||cccccc||}
    \hline\hline
     \# & \multicolumn{6}{|c||}{Convergence rates $O(h^r)$, $r=$} \\
    triangles & $\|\nabla_d e_h\|$ & $\|e_0\|$ & $\|e_b\|$ & $\|\nabla_d u_h - \nabla u\|$ & $\|u_0-u\|$ & $\|e_0\|_{\infty}$ \\ \hline\hline
    268 & 0.0604 &   0.2446 &  0.3229 &   0.1084 &   0.8461 &   0.0239 \\ \hline
    300 & 0.0750 &   0.2623 &  0.3489 &   0.1206 &   0.8699 &   0.0373 \\ \hline
    332 & 0.0888 &   0.2818 &  0.3772 &   0.1329 &   0.8912 &   0.0487 \\ \hline
    364 & 0.1020 &   0.3031 &  0.4079 &   0.1454 &   0.9099 &   0.0586 \\ \hline
    396 & 0.1148 &   0.3266 &  0.4411 &   0.1581 &   0.9260 &   0.0673 \\ \hline
    428 & 0.1273 &   0.3522 &  0.4766 &   0.1711 &   0.9396 &   0.0749 \\ \hline
    460 & 0.1396 &   0.3802 &  0.5145 &   0.1843 &   0.9509 &   0.0817 \\ \hline
    492 & 0.1519 &   0.4105 &  0.5548 &   0.1978 &   0.9602 &   0.0878 \\ \hline
    524 & 0.1641 &   0.4432 &  0.5972 &   0.2117 &   0.9678 &   0.0932 \\ \hline\hline
  \end{tabular}
\end{table}

%%%%%%%%%%%%%%%%%%%%%%%%%%%%%%%%%%%%%%%%%%%%%%%%%%%%%%%%%%%%%%%%%%%%%%%%%%%%%%%%%
\subsection{Case 5: An Anisotropic Problem} Consider a two dimensional
anisotropic problem defined in the square domain $\Omega = (0,1)^2$
as follows
$$
-\nabla\cdot(\calA \nabla u) = f,
$$
where the diffusive coefficient is given by
$$
\calA = \begin{bmatrix}k^2 & 0 \\ 0 & 1\end{bmatrix},\qquad \textrm{for }k\neq 0.
$$
We chose a function $f$ and a Dirichlet boundary condition so that
the exact solution is given by $u(x,y) = \sin(2\pi x) \sin(2k\pi
y)$. In applying the weak Galerkin method, we use an anisotropic
triangular mesh that was constructed by first dividing the domain
into $kn\times n$ sub-rectangles, and then splitting each rectangle
into two triangles by connecting a diagonal line. The characteristic
mesh size is $h = 1/n$. We tested two cases with $k = 3$ and $k =
9$. The results are reported in Tables \ref{tab:anistropic1} and
\ref{tab:anistropic2}. The tables show optimal rates of convergence
for the weak Galerkin approximation in various metrics. The
numerical experiment indicates that the weak Galerkin method can
handle anisotropic problems and meshes without any trouble.

\begin{table}
  \caption{Case 5. Convergence rate for the anisotropic problem with $k=3$.}
  \label{tab:anistropic1}
\begin{tabular}{||c||cccccc||}
    \hline\hline
   $h$ & $\|\nabla_d e_h\|$ & $\|e_0\|$ & $\|e_b\|$ & $\|\nabla_d u_h - \nabla u\|$ & $\|u_0-u\|$ & $\|e_0\|_{\infty}$ \\
    \hline\hline
   1/8   & 1.48e+0  &  1.95e-02 & 4.61e-02  &  2.70e+0  &  1.29e-01 &  4.13e-02 \\ \hline
   1/16  & 7.39e-01 &  5.11e-03 & 1.16e-02  &  1.35e+0  &  6.53e-02 &  1.06e-02 \\ \hline
   1/32  & 3.69e-01 &  1.29e-03 & 2.92e-03  &  6.80e-01 &  3.27e-02 &  2.67e-03 \\ \hline
   1/64  & 1.84e-01 &  3.24e-04 & 7.33e-04  &  3.40e-01 &  1.63e-02 &  6.68e-04 \\ \hline
   1/128 & 9.23e-02 &  8.12e-05 & 1.83e-04  &  1.70e-01 &  8.18e-03 &  1.66e-04 \\ \hline\hline
   $\begin{matrix}O(h^r)\\r=\end{matrix}$ & 1.0010 &   1.9793  &  1.9942 &   0.9972 &   0.9975  &  1.9906 \\ \hline\hline
   \end{tabular}
\end{table}

\begin{table}
  \caption{Case 5. Convergence rate for the anisotropic problem with $k=9$.}
  \label{tab:anistropic2}
\begin{tabular}{||c||cccccc||}
    \hline\hline
   $h$ & $\|\nabla_d e_h\|$ & $\|e_0\|$ & $\|e_b\|$ & $\|\nabla_d u_h - \nabla u\|$ & $\|u_0-u\|$ & $\|e_0\|_{\infty}$ \\
    \hline\hline
   1/4   & 7.98e+0  &  6.80e-02 & 2.93e-01  &  1.58e+1  &  2.52e-01 & 1.49e-01  \\ \hline
   1/8   & 3.89e+0  &  2.07e-02 & 7.44e-02  &  8.18e+0  &  1.30e-01 & 4.22e-02  \\ \hline
   1/16  & 1.91e+0  &  5.43e-03 & 1.88e-02  &  4.12e+0  &  6.53e-02 & 1.09e-02  \\ \hline
   1/32  & 9.54e-01 &  1.37e-03 & 4.72e-03  &  2.06e+0  &  3.27e-02 & 2.74e-03  \\ \hline
   1/64  & 4.76e-01 &  3.44e-04 & 1.18e-03  &  1.03e+0  &  1.63e-02 & 6.84e-04  \\ \hline\hline
   $\begin{matrix}O(h^r)\\r=\end{matrix}$ & 1.0161  &  1.9160 &   1.9897 &   0.9857 &   0.9883 &   1.9492 \\ \hline\hline
   \end{tabular}
\end{table}

%%%%%%%%%%%%%%%%%%%%%%%%%%%%%%%%%%%%%%%%%%%%%%%%%%%%%%%%%%%%%%%%%%%%%%%%%%%%%%%%%
\subsection{Case 6: A Three-Dimensional Model Problem} The final
test problem is a three dimensional Laplace equation defined on
$\Omega=(0,1)^3$, with a Dirichlet boundary condition and an exact
solution given by $u=\sin(2\pi x)\sin(2\pi y)\sin(2\pi z)$. The
purpose of this test problem is to examine the convergence rate of
the cubic $(Q_0(K),\,Q_0(F),\, RT_0(K))$ element. The results are
reported in Table \ref{tab:3d1}.

In addition to the optimal rates of convergence as shown in Table
\ref{tab:3d1}, on can also see a superconvergence for $\|\nabla_d
e_h\|$. The same result is anticipated for 2D rectangular elements.
It is left to interested readers for a further investigation,
especially for model problems with variable coefficients.

\begin{table}
  \caption{Case 6. Convergence rate for a 3D model problem with smooth solution.}
  \label{tab:3d1}
 \begin{tabular}{||c||cccccc||}
    \hline\hline
   $h$ & $\|\nabla_d e_h\|$ & $\|e_0\|$ & $\|e_b\|$ & $\|\nabla_d u_h - \nabla u\|$ & $\|u_0-u\|$ & $\|e_0\|_{\infty}$ \\
    \hline\hline
   1/8  & 1.85e-01 &  1.62e-02 & 4.27e-02  &  1.22e+00 &  1.34e-01 & 3.63e-02  \\ \hline
   1/12 & 8.53e-02 &  7.69e-03 & 1.94e-02  &  8.19e-01 &  9.14e-02 & 1.96e-02  \\ \hline
   1/16 & 4.86e-02 &  4.42e-03 & 1.10e-02  &  6.15e-01 &  6.89e-02 & 1.18e-02  \\ \hline
   1/20 & 3.13e-02 &  2.85e-03 & 7.07e-03  &  4.92e-01 &  5.52e-02 & 7.78e-03  \\ \hline\hline
   $\begin{matrix}O(h^r)\\r=\end{matrix}$ & 1.9389 &   1.8984 &  1.9618  &   0.9914  &  0.9737 &   1.6779 \\ \hline\hline
   \end{tabular}
\end{table}

\newpage

%%%%%%%%%%%%%%%%%%%%%%%%%%%%%%%%%%%%%%%%%%%%%%%%%%%%%%%%%%%%%%%%%%%%%%%%%%%%%%%%%


\begin{thebibliography}{99}

\bibitem{AdamsFour_BookAP_2003}
R.A.~Adams and J.J.F.~Fournier,
\textit{Sobolev Spaces},
Academic Press, 2nd ed., 2003.

\bibitem{ArnBrezCockMari_SINUM_2001}
D.~Arnold, F.~Brezzi, B.~Cockburn, and D.~Marini,
\textit{Unified analysis of discontinuous Galerkin methods
for elliptic problems},
SIAM J. Numer. Anal., {\bf 39}(2002), pp.~1749--1779.

\bibitem{Bahriawati}
C.~Bahriawati and C.~Carstensen,
\textit{Three Matlab implementations of the lowest order Raviart-Thomas MFEM with a posteriori error control},
Comp. Meth. Appl. Math., {\bf 5}(2005), pp.~333--361.

\bibitem{BauOden_CMAME_1999}
C.E.~Baumann and J.T.~Oden,
\textit{A discontinuous hp finite element method
for convection-diffusion problems},
Comput. Meth. Appl. Mech. Engrg., {\bf 175}(1999), pp.~311--341.

\bibitem{bdm}
F. Brezzi, J. Douglas, and L. Marini,
\textit{Two families of mixed finite elements for second order elliptic problems},
Numer. Math., {\bf 47}(1985), pp.~217--235.

\bibitem{BerVer_NM_2000}
C.~Bernardi and R.~Verf\"{u}rth,
\textit{Adaptive finite element methods for elliptic equations
with non-smooth coefficients},
Numer. Math., {\bf 85}(2000), pp.~579--608.

\bibitem{BrenOwensSung_ETNA_2008}
S.C.~Brenner, L.~Owens, and L.Y.~Sung,
\textit{A weakly over-penalized symmetric interior penalty method},
Electronic Transactions on Numerical Analysis (ETNA),
{\bf 30}(2008), pp.~107--127.

\bibitem{BrenScott_BookSpringer_2008}
S.C.~Brenner and L.R.~Scott,
\textit{The mathematical theory of finite element methods},
Springer, 3rd ed., 2008.

\bibitem{BrezFort_BookSpringer_1991}
F.~Brezzi and M.~Fortin,
\textit{Mixed and hybrid finite element methods},
Springer-Verlag, 1991.

\bibitem{Ciarlet78}
{\sc P.G. Ciarlet},
\textit{The finite element method for elliptic problems},
North-Holland, Amsterdam, 1978.

\bibitem{CockGopaLaza_SINUM_2009}
B.~Cockburn, J.~Gopalakrishnan, and R.~Lazarov,
\textit{Unified hybridization of
discontinuous Galerkin, mixed, and continuous Galerkin methods
for second order elliptic problems},
SIAM J. Numer. Anal., {\bf 47}(2009), pp.~1319--1365.

\bibitem{CockKarnShu_BookSpringer_2000}
B.~Cockburn, G.E.~Karniadakis, and C.-W.~Shu,
\textit{Discontinuous Galerkin Methods:
Theory, Computation and Applications},
Lect. Notes Comput. Sci. Engrg. 11,
Springer-Verlag, New York, 2000.

\bibitem{CrouRavi_RAIRO_1973}
M.~Crouzeix and P.~Raviart,
\textit{Conforming and nonconforming finite element methods
for solving the stationary Stokes equations},
RAIRO, Anal. Numer. {\bf 3}(1973), pp.~33--75.

\bibitem{EpshRiv_JCAM_2007}
Y.~Epshteyn and B.~Riviere,
\textit{Estimation of penalty parameters
for symmetric interior penalty Galerkin methods},
J. Comput. Appl. Math., {\bf 206}(2007), pp.~843--872.

\bibitem{Johnson_BookDover_2009}
C.~Johnson,
\textit{Numerical solution of partial differential equations
by the finite element method},
Dover, 2009.

\bibitem{Kell_ApplAnal_1975}
R.B.~Kellog,
\textit{On the Poisson equation with intersecting interfaces},
Appl. Anal., {\bf 4}(1976), pp.~101--129.

\bibitem{LiuMuYe_PrepJCAM_2011}
J.~Liu, L.~Mu, and X.~Ye,
\textit{Convergence of the discontinuous finite volume method
for elliptic problems with minimal regularity},
Preprint submitted to J. Comput. Appl. Math, (2011).

\bibitem{LiuYang_PrepNMPDE_2011}
J.~Liu and M.~Yang,
\textit{A weakly over-penalized finite volume element method
for elliptic problems},
Preprint submitted to Numer. Meth. PDEs, (2011).

\bibitem{MargeVassi_SISC_1994}
S.D.~Margenov and P.S.~Vassilevski,
\textit{Algebraic multilevel preconditioning
of anisotropic elliptic problems},
SIAM J. Sci. Comput., {\bf 15}(1994), pp.~1026--1037.

\bibitem{MorinNocheSie_SIREV_2002}
P.~Morin, R.H.~Nochetto, and K.G.~Siebert,
\textit{Convergence of adaptive finite element methods},
SIAM Rev., {\bf 44}(2002), pp.~631--658.

\bibitem{rt}
P. Raviart and J. Thomas,
\textit{A mixed finite element method for second order elliptic problems},
Mathematical Aspects of the Finite Element Method, I. Galligani, E. Magenes, eds., Lectures Notes in Math. 606, Springer-Verlag, New York,  1977.

\bibitem{SunShuyu_UTAustin_PhD2003}
S.~Sun,
\textit{Discontinuous Galerkin methods for reactive transport in porous media},
Ph.D. dissertation, The University of Texas at Austin, 2003.

\bibitem{WangYe_PrepSINUM_2011}
J.~Wang and X.~Ye,
\textit{A weak Galerkin finite element method
for second-order elliptic problems},
arXiv:1104.2897v1 [math.NA].

\bibitem{WangYe_PrepSINUM_2011b}
J.~Wang and X.~Ye,
\textit{A weak Galerkin finite element method for Stokes problems},
Preprint, 2011.

\bibitem{RivWheeGir_SINUM_2001}
B.~Rivi\'{e}re, M.F.~Wheeler, and V.~Girault,
\textit{A priori error estimates for finite element methods
based on discontinuous approximation spaces for elliptic problems},
SIAM J. Numer. Anal., {\bf 39}(2001), pp.~902--931.

\bibitem{WihRiv_JSC_2011}
T.P.~Wihler and B.~Rivi\'{e}re,
\textit{Discontinuous Galerkin methods for second-order elliptic PDE
with low-regularity solutions},
J. Sci. Comput., {\bf 46}(2011), pp.~151--165.

\end{thebibliography}
\end{document}